\documentclass{article}
\usepackage{amssymb,amsmath,graphicx,ucs}
\usepackage[utf8x]{inputenc}

\newtheorem{lemma}{Lemma}
\newtheorem{theorem}{Theorem}

\newenvironment{proof}{\noindent \emph{Proof. }}{\hfill \hbox{\rlap{$\sqcap$}$\sqcup$}\\}

\title{The Ammann-Beenker Tilings Revisited}

\author{Nicolas Bédaride\footnote{LATP, Univ. Aix-Marseille, nicolas.bedaride@latp.univ-mrs.fr} \and Thomas Fernique\footnote{LIPN, CNRS \& Univ. Paris 13, thomas.fernique@lipn.univ-paris13.fr}}

\date{}

\begin{document}

\maketitle

\begin{abstract}
This paper introduces two tiles whose tilings form a one-parameter family of tilings which can all be seen as digitization of two-dimensional planes in the four-dimensional Euclidean space.
This family contains the Ammann-Beenker tilings as the solution of a simple optimization problem.
\end{abstract}

\section{Introduction}

Having decided to retile your bathroom this week-end, you go to your favorite retailer of construction products.
There, you see a unusual special offer on two strange notched tiles (Fig.~\ref{fig:tiles}): ``Pay the squares cash, get the rhombi for free!''\\

\begin{figure}[hbtp]
\centering
\includegraphics[width=0.4\textwidth]{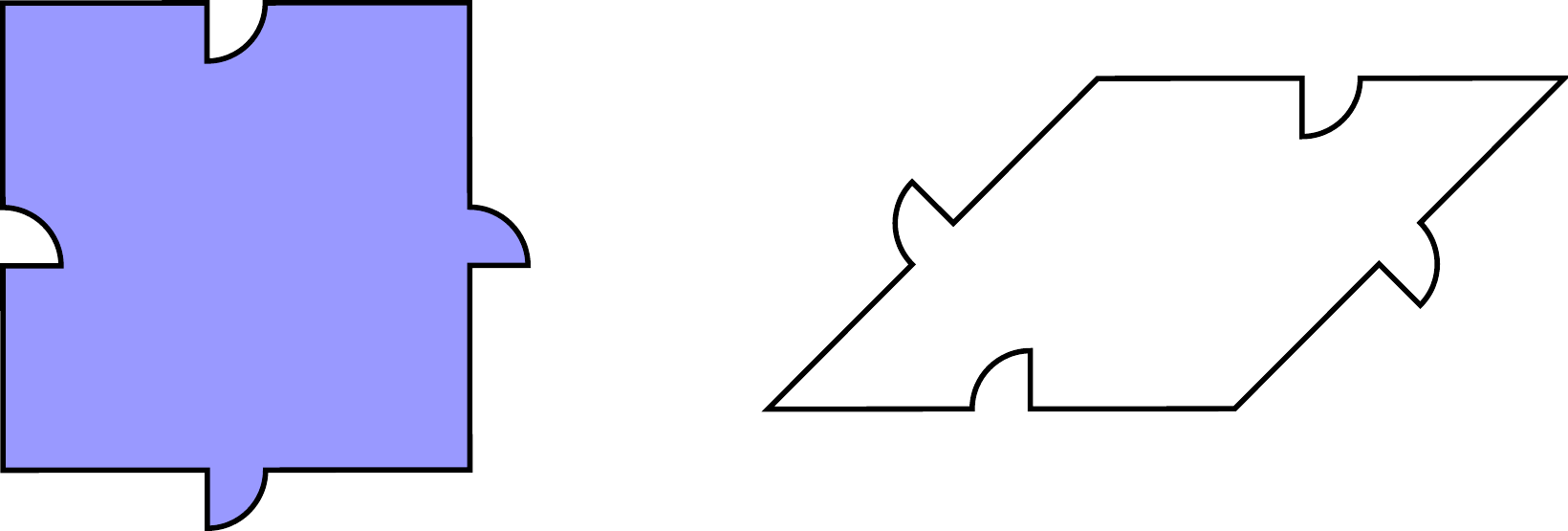}
\caption{Two notched tiles.}
\label{fig:tiles}
\end{figure}

Fearing that this might be a scam, you try to figure out how your bathroom could be tiled at little cost.
After careful consideration, you see that the possible tilings are exactly those where any two rhombi adjacent or connected by lined up squares have different orientations (see Fig.~\ref{fig:alternance}).
In particular, rhombi only do not tile, so you would have to buy at least some squares.
You could of course tile with squares only (on a grid), but this would be missing this special offer!\\

\begin{figure}[hbtp]
\centering
\includegraphics[width=0.9\textwidth]{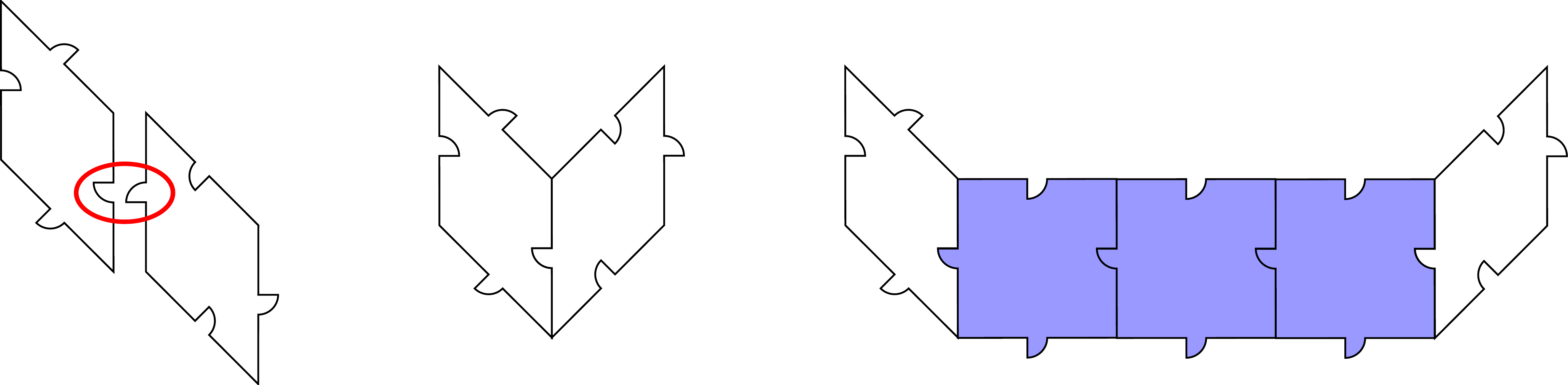}
\caption{Two rhombi match only if they have different orientations.
This still holds with lined up squares between them, since those just carry the notching.}
\label{fig:alternance}
\end{figure}

We will show that the cheapest (if not the simplest) way to tile your bathroom is to form a non-periodic tiling, namely an {\em Ammann-Beenker tiling}.
Furthermore, we will show that the set of all possible tilings form a one-parameter family of tilings which can all be seen as digitization of two-dimensional planes in the four-dimensional Euclidean space.
Fig.~\ref{fig:tilings} depicts some possible tilings, with the rightmost one being an Ammann-Beenker tiling.\\

This is of course not only of interest to tile bathrooms, but it could provide a new insight into the theory of quasicrystals.
Indeed, digitizations of irrational planes in higher dimensional spaces (also called {\em projection tilings}) are a common model of quasicrystals, and the above results give an example of how very simple local constraints can enforce long range order, with the non-periodicity simply coming from tile proportions.
In particular, slight variations of tile proportions around those of a non-periodic tiling can lead to close periodic tilings, reminding approximants of quasicrystals.\\

\begin{figure}[hbtp]
\centering
\includegraphics[width=0.31\textwidth]{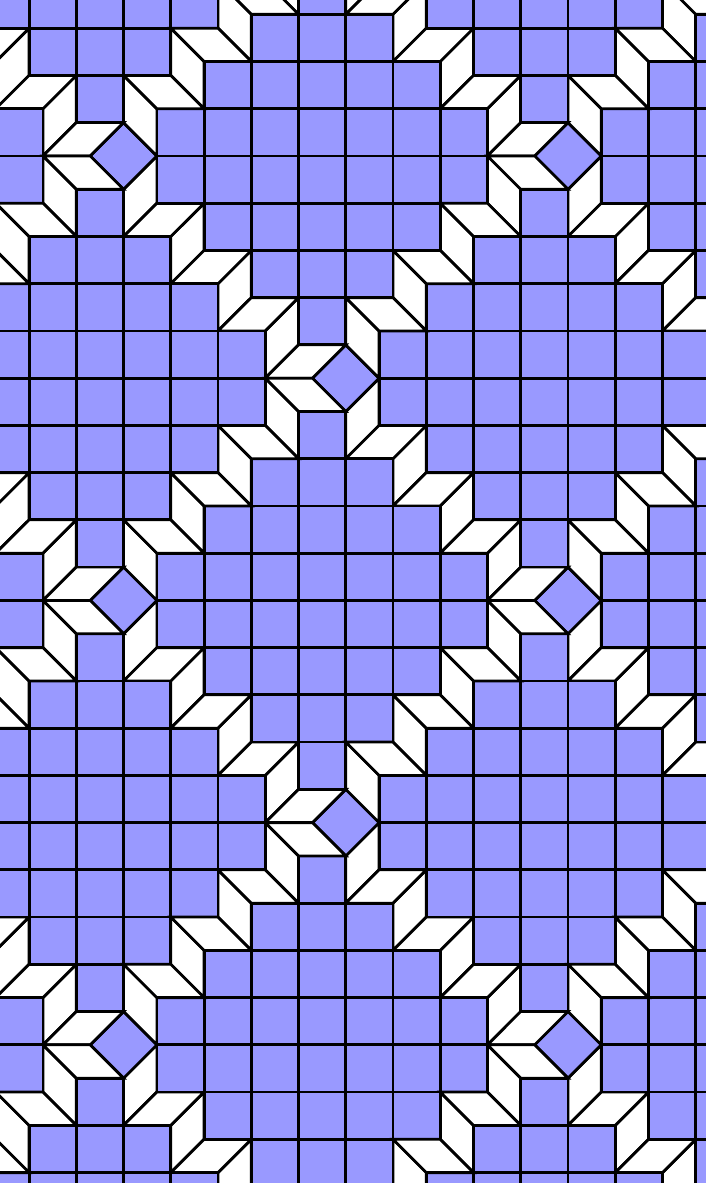}
\hfill
\includegraphics[width=0.31\textwidth]{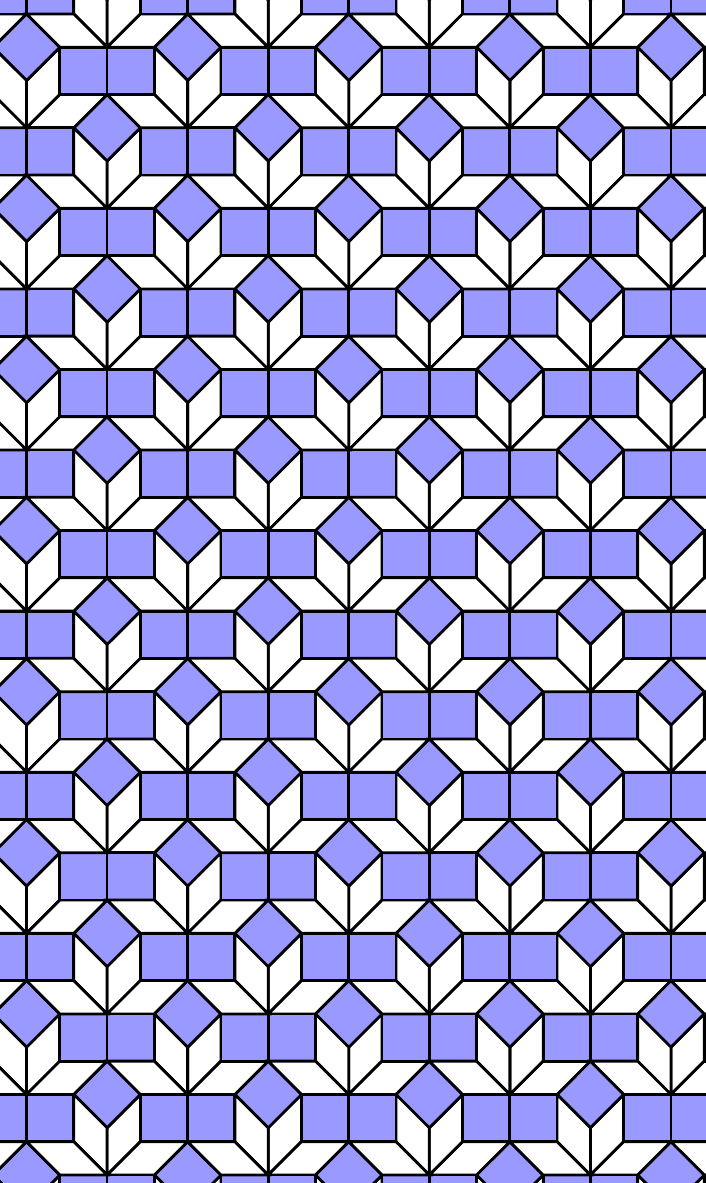}
\hfill
\includegraphics[width=0.31\textwidth]{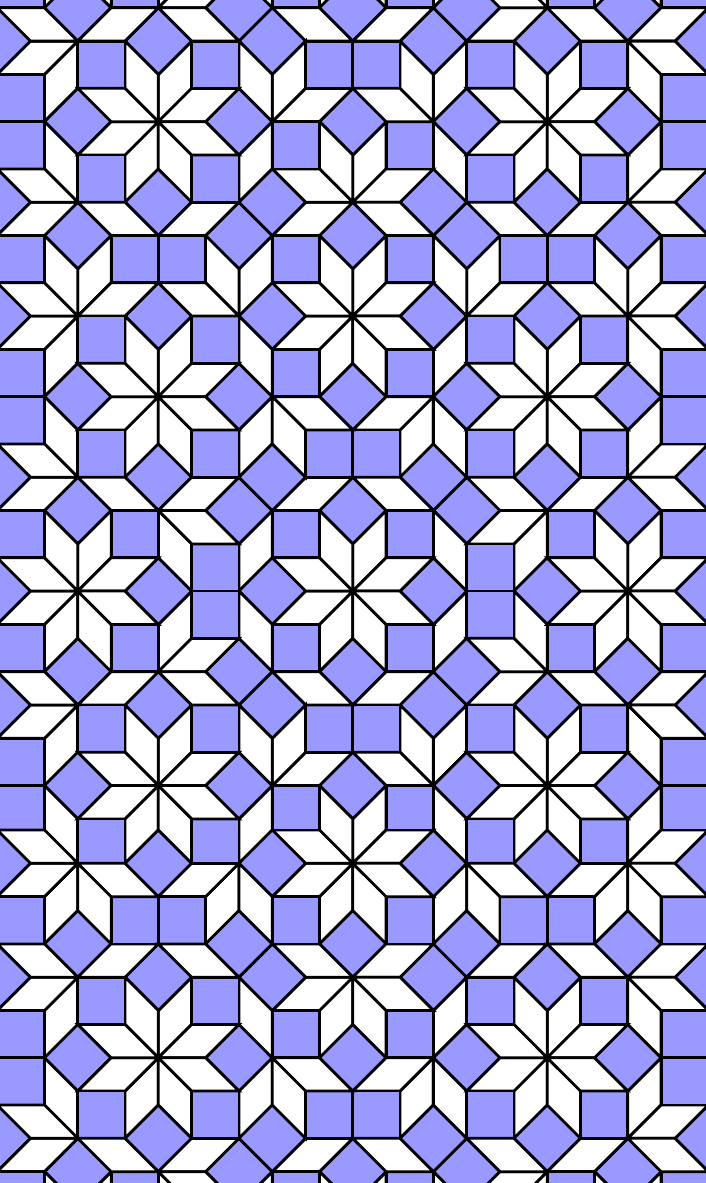}
\caption{Three different possible tilings (notching are not depicted).}
\label{fig:tilings}
\end{figure}

The rest of the paper is organized as follows.
Section \ref{sec:ammann_beenker} briefly recalls the history of Ammann-Beenker tilings.
Sections \ref{sec:tilings} and \ref{sec:shadows} introduce the main notions, Section \ref{sec:algebra} makes a simple but powerful connection with classic results of algebraic geometry, and the technical part of our proof is exposed in Section \ref{sec:planarity}.
We conclude in Section \ref{sec:conclusion} by formally stating our main result (Theorem~\ref{th:main}).

\section{Ammann-Beenker tilings}\label{sec:ammann_beenker}

Ammann-Beenker tilings are non-periodic tilings of the plane by a square and a rhombus with a $45^\circ$ angle.
Enjoying a (local) $8$-fold symmetry, they became a popular model of the $8$-fold quasicrystals \cite{octogonal_quasicrystal}.
They were introduced by Ammann in the 1970s and Beenker in 1982, independantly and from different viewpoints.\\

On the one hand, Ammann defined these tilings as the ones that can be formed by two specific notched tiles and a ``key'' tile, with the non-periodicity deriving from the hierarchical structure enforced by the notching.
This can be compared to the first (and concomitant) definition of Penrose tilings \cite{penrose}.\\

On the other hand, following the algebraic approach of de Bruijn for Penrose tilings \cite{debruijn}, Beenker defined these tilings, that he called {\em Grid-Rhombus}, as digitizations of parallel planes in $\mathbb{R}^4$, with the non-periodicity deriving from the irrationality of the slope of these planes \cite{beenker}.
Unfortunately, Beenker was unaware of the work of the amateur mathematician Ammann, published only some years later \cite{GS}, and he was unable to find notched tiles which can form only these tilings.
Instead, he introduced the notching of Fig.~\ref{fig:tiles}, calling {\em Arrowed-Rhombus} the tilings which can be formed and proving that they strictly contain the Grid-Rhombus tilings.\\

To conclude this short review, let us mention that Ammann-Beenker tilings cannot be characterized by their local patterns, that is, for any $r\geq 0$, there exists a tiling whose patterns of radius $r$ all appear in an Ammann-Beenker tiling but which is not itself an Ammann-Beenker tiling \cite{burkov}.
Suitable notchings of tiles must thus carry some information over arbitrarily long distances!

\section{Octogonal tilings and planarity}\label{sec:tilings}

Let $\vec{v}_1,\ldots,\vec{v}_4$ be pairwise non-colinear unitary vector of the Euclidean plane.
We define the six rhombi $\{\lambda\vec{v}_i+\mu\vec{v}_j~|~0\leq\lambda,\mu\leq 1\}$, for $1\leq i<j\leq 4$, and we call {\em octogonal tiling} any covering of the Euclidean plane by translated rhombi, where rhombi can intersect only on a vertex or along a complete edge (Fig.~\ref{fig:tilings}).\\

Let $\vec{e}_1,\ldots,\vec{e}_4$ be the canonical basis of $\mathbb{R}^4$.
A {\em lift} of an octogonal tiling is obtained by mapping its rhombi onto faces of unit hypercubes $\mathbb{Z}^4$ so that any two rhombi adjacent along $\vec{v}_k$ are mapped onto unit faces adjacent along $\vec{e}_k$.
This is a two-dimensional surface of $\mathbb{R}^4$ which is uniquely defined up to translation.\\

An octogonal tiling is said to be {\em planar} if there are a two-dimensional plane $E\subset \mathbb{R}^4$ and $t\geq 1$ such that it can be lifted into the ``slice'' $E+[0,t]^4$.
The plane $E$ is called its {\em slope} and the smallest suitable $t$ its {\em thickness} (both are unique).
A planar octogonal tiling can be seen as a digitization of its slope.\\

For example, the Ammann-Beenker tilings are the planar octogonal tilings of thickness one whose slope is generated by $(\cos\frac{k\pi}{4})_{0\leq k<4}$ and $(\sin\frac{k\pi}{4})_{0\leq k<4}$.\\

Planar octogonal tilings form a subclass of the so-called {\em projection tilings}.
Those of thickness one are periodic for a rational slope, {\em quasiperiodic} otherwise, {\em i.e.}, any pattern of radius $r$ which appears somewhere in a tiling reappears in this tiling at a distance uniformly bounded in $r$.
This perfect order weakens when the thickness increases, but the long range order nevertheless persists.

\section{Shadows and subperiods}\label{sec:shadows}

The {\em $k$-th shadow} of an octogonal tiling is the orthogonal projection of its lift along $\vec{e}_k$.
Formally, a $k$-th shadow is a lift of an octogonal tiling, {\em i.e.}, a two-dimensional surface of $\mathbb{R}^4$, but since it does not contain unit faces with the edge $\vec{e}_k$, it can be convenient to see it as a two-dimensional surface of $\mathbb{R}^3$.\\

A {\em period} of a shadow is a translation vector leaving invariant the shadow.
The {\em subperiods} of an octogonal tilings are the periods of its shadows.\\

Fig.~\ref{fig:shadows} depicts the fourth shadows of the tilings of Fig.~\ref{fig:tilings}: they are periodic.
Actually, the alternation of rhombus orientations in these tilings, discussed in the introduction, precisely enforces a period for each shadow.
Formally, one checks that with $\vec{v}_k=\textrm{e}^{\mathrm{i}\frac{k\pi}{4}}$ (complex notation) for $1\leq k\leq 4$, the $k$-th shadow of any such tiling admits the period $\vec{p}_k$ defined by
$$
\vec{p}_1=\vec{e}_2-\vec{e}_4,
\qquad
\vec{p}_2=\vec{e}_1+\vec{e}_3,
\qquad
\vec{p}_3=\vec{e}_2+\vec{e}_4,
\qquad
\vec{p}_4=\vec{e}_1-\vec{e}_3.
$$

\begin{figure}[hbtp]
\centering
\includegraphics[width=0.31\textwidth]{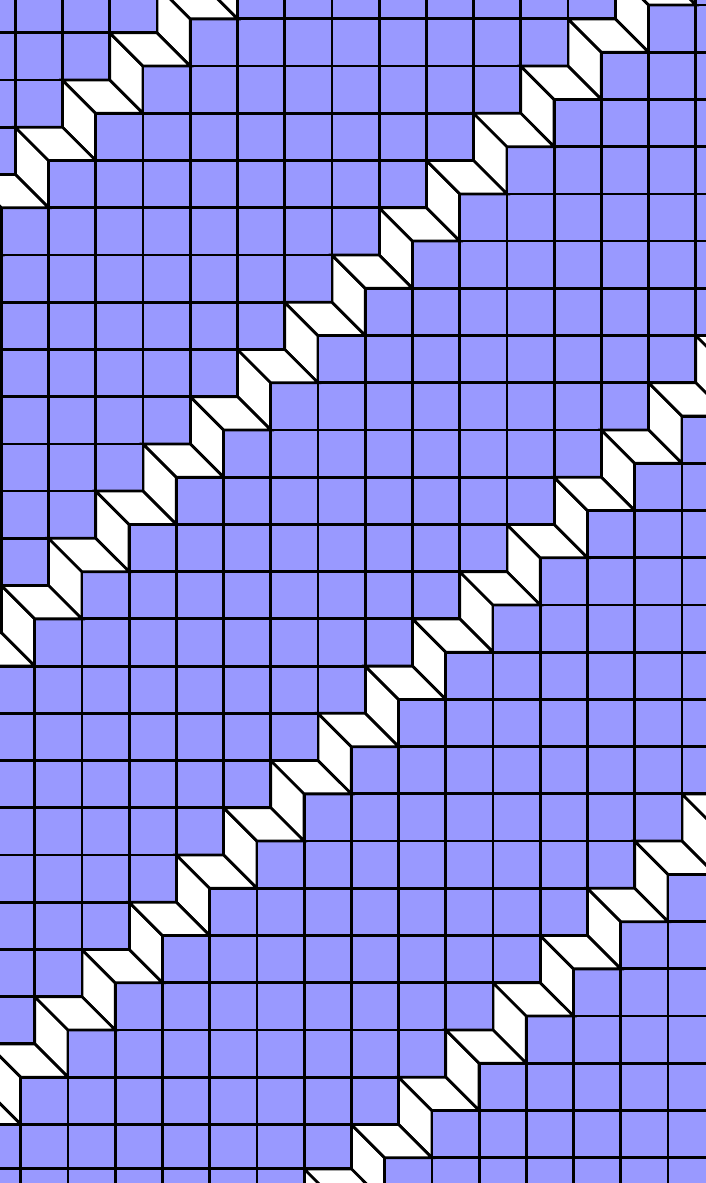}
\hfill
\includegraphics[width=0.31\textwidth]{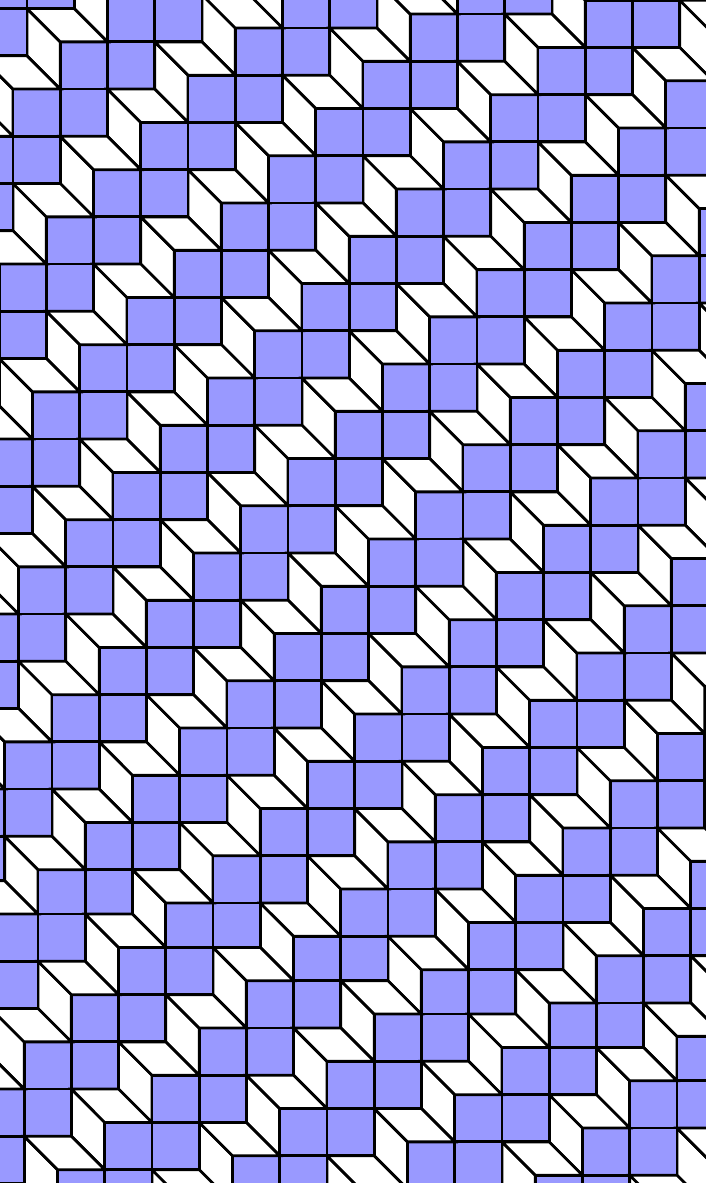}
\hfill
\includegraphics[width=0.31\textwidth]{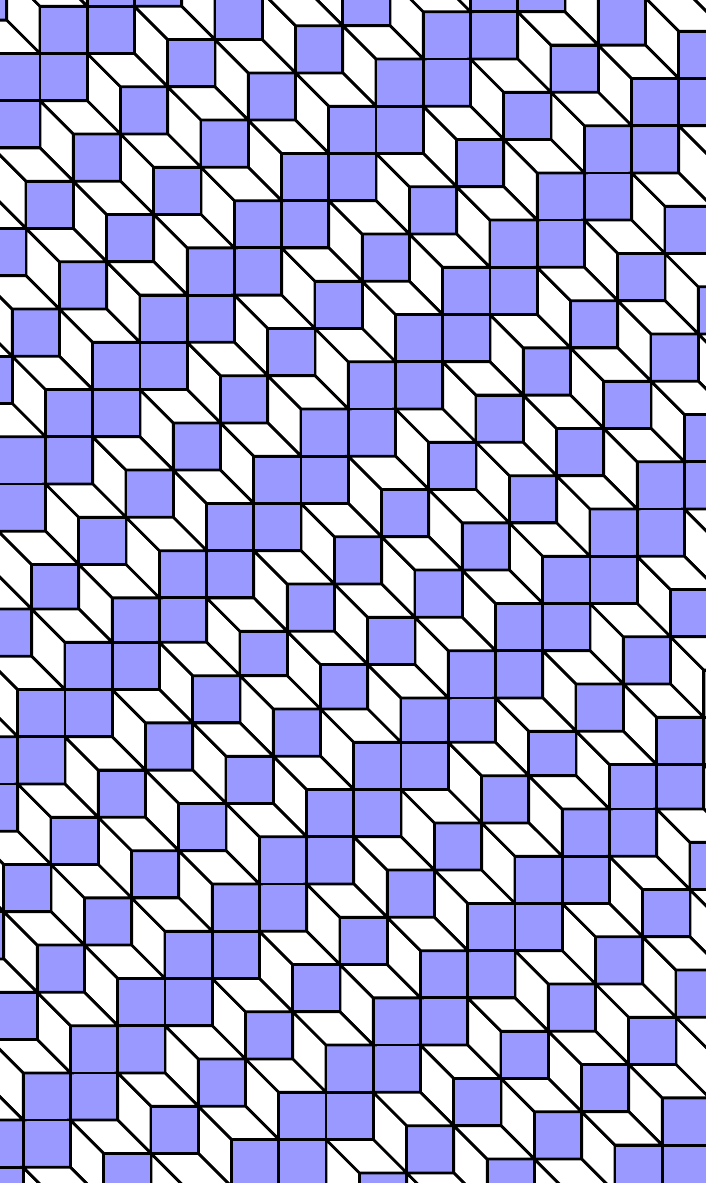}
\caption{Shadows of the tilings depicted on Fig.~\ref{fig:tilings}.}
\label{fig:shadows}
\end{figure}

\section{Grassmann coordinates and Plücker relations}\label{sec:algebra}

First, recall (see, {\em e.g.}, \cite{HP}, chap. 7) that a two-dimensional plane $E$ of $\mathbb{R}^4$ generated by $(u_1,u_2,u_3,u_4)$ and $(v_1,v_2,v_3,v_4)$ has for {\em Grassmann coordinates} the numbers $G_{ij}=u_iv_j-u_jv_i$, $1\leq i<j\leq 4$.
Theses coordinates are unique up to a common multiplicative constant ; one writes $E=(G_{12},G_{13},G_{14},G_{23},G_{24},G_{34})$.
Conversely, any $G_{ij}$'s not all equal to zero are the Grassmann coordinates of some two-dimensional plane of $\mathbb{R}^4$ if and only if they satisfy the {\em Plücker relation}
$$
G_{12}G_{34}=G_{13}G_{24}-G_{14}G_{23}.
$$

Then, it is not hard to see that if the $l$-th shadow of a planar octogonal tiling of slope $E$ admits a period $(p,q,r)$, then the Grassmann coordinates satisfy
$$
pG_{jk}-qG_{ik}+rG_{ij}=0,
$$
where $l\notin\{i,j,k\}$.
Indeed, if $E$ is generated by $(u_1,u_2,u_3,u_4)$ and $(v_1,v_2,v_3,v_4)$, then the $l$-th shadow can be seen as a digitization of the plane of $\mathbb{R}^3$ generated by $(u_i,u_j,u_k)$ and $(v_i,v_j,v_k)$.
If $(p,q,r)$ is a period of this plane, it belongs to this plane and thus has a zero dot product with the normal vector $(G_{jk},-G_{ik},G_{ij})$.\\

One can also use shadows to show that in any planar octogonal tiling of slope $E$, the ratio between the proportions of tiles with edges $\vec{v}_i$ and $\vec{v}_j$ and those with edges $\vec{v}_k$ and $\vec{v}_l$ is $|G_{ij}/G_{kl}|$.\\

Now, consider a tiling by tiles of Fig.~\ref{fig:tiles}: it is octogonal up to the notching.
If we assume that it is planar, then its subperiods yield
$$
G_{23}=G_{34},
\qquad
G_{14}=G_{34},
\qquad
G_{12}=G_{14},
\qquad
G_{12}=G_{23},
$$
and plugging this into the Plücker relation, a short computation shows that the slope must be one of the planes
$$
E_0:=(0,0,0,0,1,0),\qquad
E_{t\neq 0}:=(1,t,1,1,2/t,1),\qquad
E_{\infty}:=(0,1,0,0,0,0).
$$
Conversely, any planar octogonal tiling with one of these slopes and thickness one satisfies the alternation of rhombi orientations (two rhombi with the same orientation would not fit into the slice), thus can be tiled by the tiles of Fig.~\ref{fig:tiles}.\\

For example, the tilings of Fig.~\ref{fig:tilings} have respective slope $E_{1/4}$, $E_1$ and $E_{\sqrt{2}}$.
In the latter case, which is an Ammann-Beenker tiling, there is thus $\sqrt{2}$ rhombi for each square (since the square area is $\sqrt{2}$ times the rhombus area, each tile covers exactly half of the plane).
Tilings by squares only have slope $E_0$ or $E_{\infty}$.\\

\noindent However, nothing yet ensures that tilings by Fig.~\ref{fig:tiles} tiles are indeed planar!

\section{Planarity}\label{sec:planarity}

\begin{lemma}\label{lem:planarity}
Fig.~\ref{fig:tiles} tiles form only planar tilings of uniformly bounded thickness.
\end{lemma}

\begin{proof}
Let $E:=E_{\sqrt{2}}$.
One checks that the orthogonal projection of the $\vec{e}_i$'s onto $E$ are pairwise non-colinear vectors.
Let us identify $E$ with the two-dimensional Euclidean plane and the above projections (up to rescaling) with the $\vec{v}_i$'s which define the tiles, so that the orthogonal projection onto $E$ is a homeomorphism from any lift of any tiling of the Euclidean plane by these tiles onto $E$.
Let $\mathcal{T}$ be such a tiling and $\mathcal{S}$ be a lift of it.
Define
$$
\vec{q}_1=\vec{p}_1+\sqrt{2}\vec{e}_1,
\qquad
\vec{q}_2=\vec{p}_2+\sqrt{2}\vec{e}_2,
\qquad
\vec{q}_3=\vec{p}_3+\sqrt{2}\vec{e}_3.
\qquad
\vec{q}_4=\vec{p}_4-\sqrt{2}\vec{e}_4.
$$
Those are pairwise non-colinear vectors of $E$.
Let also $\vec{r}_i$ be obtained by changing $\sqrt{2}$ in $-\sqrt{2}$ in $\vec{q}_i$.
The $\vec{r}_i$'s are pairwise non-colinear vectors of $E':=E_{-\sqrt{2}}$.
One checks that $E$ and $E'$ are orthogonal planes, so that there exist two real functions $z_1$ and $z_2$ defined on $E$ such that the lift $\mathcal{S}$ is the image of $E$ under
$$
\rho~:~\vec{x}\mapsto\vec{x}+z_1(\vec{x})\vec{r}_1+z_2(\vec{x})\vec{r}_2.
$$
Let us show that the subperiods of $\mathcal{T}$ enforce the map $\rho$ to be almost linear.
Let $\pi_i$ denotes the orthogonal projection along $\vec{e}_i$.
One has $\pi_i(\vec{q}_i)=\pi_i(\vec{r}_i)=\vec{p}_i$.
For any $\vec{x}\in E$, the plane $\pi_i(\vec{x}+E')$ intersects the shadow $\pi_i(\mathcal{S})$ along the curve
$$
\mathcal{C}_i(\vec{x})=\{\pi_i(\vec{x})+z_1(\vec{x}+\lambda\vec{q}_i)\pi_i(\vec{r}_1)+z_2(\vec{x}+\lambda\vec{q}_i)\pi_i(\vec{r_2})~|~\lambda\in\mathbb{R}\}.
$$
Since both $\pi_i(\mathcal{S})$ and $\pi_i(\vec{x}+E')$ are $\vec{p}_i$-periodic, so is $\mathcal{C}_i(\vec{x})$.
In particular, it stays at bounded distance from some line directed by $\vec{p}_i$.
For $i=1$, since $\pi_1(\vec{r}_1)=\vec{p}_1$, this ensures that $\lambda\mapsto z_2(\vec{x}+\lambda\vec{q}_1)$ is uniformly bounded.
In other words, $z_2$ has bounded fluctuations in the direction $\vec{q}_1$.
Similarly, for $i=2$, $\pi_2(\vec{r}_2)=\vec{p}_2$ yields that $z_1$ has bounded fluctuations in the direction $\vec{q}_2$.
For $i=3$, one computes
$$
\vec{p}_3=-\pi_3(\vec{r}_1)-\sqrt{2}\pi_3(\vec{r}_2),
$$
what yields bounded fluctuations for $z_2-\sqrt{2}z_1$ in the direction $\vec{q}_3$.
Since $\vec{q}_1$ and $\vec{q}_2$ form a basis of $E$, let $z_i(\lambda,\mu)$ stand for $z_i(\lambda\vec{q}_1+\mu\vec{q}_2)$, $i\in\{1,2\}$, and write $f\equiv g$ if the difference of two functions $f$ and $g$ is uniformly bounded.
The bounded fluctuations of $z_1$ and $z_2$ in the directions $\vec{q}_1$ and $\vec{q}_2$ yield the existence of real functions $f$ and $g$ such that $z_2(\lambda,\mu)\equiv f(\mu)$ and $z_1(\lambda,\mu)\equiv g(\lambda)$.
Further, since $\vec{q}_3=\sqrt{2}\vec{q}_2-\vec{q}_1$, the bounded fluctuations of $z_2-\sqrt{2}z_1$ in the direction $\vec{q}_3$ yield the existence of a real function $h$ such that $(z_2-\sqrt{2}z_1)(\lambda,\mu)\equiv h(\sqrt{2}\mu-\lambda)$.
Thus
$$
f(\mu)-\sqrt{2}g(\lambda)\equiv h(\sqrt{2}\mu-\lambda).
$$
Fix $\lambda=0$ to get $f(\mu)\equiv h(\sqrt{2}\mu)$.
Fix $\mu=0$ to get $-\sqrt{2}g(\lambda)\equiv h(-\lambda)$.
Hence
$$
h(\sqrt{2}\mu)+h(-\lambda)\equiv h(\sqrt{2}\mu-\lambda).
$$
From this easily follows that $h$, hence  $f$, $g$, $z_1$, $z_2$ and $\rho$, are linear (up to bounded fluctuations).
The tiling $\mathcal{T}$ is thus planar.
The thickness ({\em i.e.}, the fluctuations of $\rho$) is uniformly bounded because the lifts are lipschitz with a constant which depends only on $E$.
\end{proof}

\section{Conclusion}\label{sec:conclusion}

The following theorem summarizes the results obtained in the sections \ref{sec:algebra} and \ref{sec:planarity}:

\begin{theorem}\label{th:main}
Fig.~\ref{fig:tiles} tiles can form only planar tilings with slope in $\{E_t\}_{t\in\mathbb{R}\cup\{\infty\}}$ and uniformly bounded thickness, and they form at least those of thickness one.
\end{theorem}

Moreover, the Ammann-Beenker tilings have the slope which maximizes the area covered by rhombi: they provide the cheapest way to tile your bathroom!
Let us make some final comments.
First, although we only prove that the thickness of tilings by Fig.~\ref{fig:tiles} tiles is uniformly bounded, we conjecture that the thickness is one, so that exactly all these tilings can be formed.
Second, note that among the tilings by Fig.~\ref{fig:tiles} tiles, Ammann-Beenker tilings are exactly (up to the thickness) those whose slope satisfies the relation $G_{13}=G_{24}$, {\em i.e.}, where the squares appear with the same frequency in their two possible orientations.
The above mentionned result of \cite{burkov} shows that this relation, although simple, cannot be enforced by local patterns: when $t$ tends towards $\sqrt{2}$, the tilings of slope $E_t$ and $E_{\sqrt{2}}$ (and thickness one) become locally indistinguishable.
Last, let us stress that, to our knowledge, this is the first example of a finite set of tiles which can form only planar tilings with infinitely many different slopes.

\paragraph{Acknowledgments} We would like to thank Thang~T.~Q.~Le for sending us the unpublished preprint \cite{le92}, which inspired the proof of Lemma~\ref{lem:planarity}.

\end{document}